\documentclass[12pt,reqno]{amsart}
\usepackage{mathrsfs}
\usepackage{amsmath,amssymb,amsthm}
\usepackage{cite}
\numberwithin{equation}{section}

\def\Z{\mathbb{Z}{\ssc\,}}

\def\C{\mathbb{C}{\ssc\,}}

\def\ssc{\scriptscriptstyle}

\def\Der{{\rm Der}}

\def \<{\langle}
\def \>{\rangle}

\def\vs{\vspace*}

\def \be{\begin{equation}\label}
\def \ee{\end{equation}}
\def \bex{\begin{example}\label}
\def \eex{\end{example}}
\def \bl{\begin{lem}\label}
\def \el{\end{lem}}
\def \bt{\begin{thm}\label}
\def \et{\end{thm}}
\def \bp{\begin{prop}\label}
\def \ep{\end{prop}}
\def \br{\begin{rem}\label}
\def \er{\end{rem}}
\def \bc{\begin{coro}\label}
\def \ec{\end{coro}}
\def \bd{\begin{de}\label}
\def \ed{\end{de}}

\def \B{\mathcal{B}}

\newtheorem{thm}{Theorem}[section]
\newtheorem{prop}[thm]{Proposition}
\newtheorem{coro}[thm]{Corollary}

\newtheorem{example}[thm]{Example}
\newtheorem{lem}[thm]{Lemma}

\newtheorem{rem}[thm]{Remark}
\newtheorem{de}[thm]{Definition}

\makeatletter
\@addtoreset{equation}{section}

\makeatother
\begin{document}

\title[Local Derivations on a Block-type Lie Algebra]{Local Derivations on a  Block-type Lie Algebra}

\author{Shiyu Wu}
\address{\it Shiyu Wu: Department of Mathematics, Shanghai Maritime University, Shanghai 201306, China} 
\email{shiyuw615@163.com}

\author{Hengyun Yang$^*$}
\address{\it Hengyun Yang: Department of Mathematics, Shanghai Maritime University, Shanghai 201306, China} 
\email{hyyang@shmtu.edu.cn}

\thanks{$^*$The corresponding author}

\begin{abstract}
We study local derivations on a Block-type Lie algebra \(\mathcal B\).
Using the structure of the derivation algebra of \(\mathcal B\) and a
series of normalizations and coefficient comparisons, we prove that
every local derivation on \(\mathcal B\) is a derivation.
\end{abstract}

\maketitle
\thispagestyle{empty}


\section{Introduction}

Infinite-dimensional Lie algebras have been studied extensively in
representation theory and related areas. Among them, the Virasoro algebra and its generalizations occupy a central position due to their rich algebraic structures and important applications. Block introduced an important class of infinite-dimensional simple Lie
algebras in~\cite{Block}. Various generalizations of these algebras,
commonly referred to as Lie algebras of Block type, have since been
studied extensively. Their derivations, automorphisms, cohomology, and
representations have been investigated in
\cite{DjokovicZhao,SuFamily,XiaYouZhou,XiaZhang}.

A local derivation is a pointwise generalization of a derivation. The
notion originated in operator algebra theory and was introduced by
Kadison~\cite{Kadison} and independently studied by Larson and Sourour
\cite{LarsonSourour}. It has since been considered for various algebraic
structures, including Lie algebras and Lie superalgebras. A fundamental problem
is to determine when every local derivation on a given algebra is necessarily
a derivation.

For Lie algebras, the answer depends strongly on the structure of the
algebra. Ayupov and Kudaybergenov~\cite{AyupovKudaybergenov} proved that
every local derivation on a finite-dimensional semisimple Lie algebra is a
derivation. In the solvable setting, both phenomena may occur: some
solvable Lie algebras admit local derivations that are not derivations,
whereas positive results hold for several important classes
\cite{AyupovKhudoyberdiyev,KudaybergenovOmirovKurbanbaev}. Local
derivations have also been studied on Borel subalgebras of
finite-dimensional simple Lie algebras~\cite{YuChen},  Witt algebras
\cite{ChenZhaoZhao}, the Lie algebra \(W(2,2)\)~\cite{WuGaoLiu}, and
locally simple Lie algebras~\cite{AyupovKudaybergenovYusupov}. Related problems have also been studied for Lie superalgebras. Positive
results are known for basic classical Lie superalgebras
\cite{ChenWangNan}, the strange Lie superalgebra
\(\mathfrak q(n)\)~\cite{ChenWang}, Cartan-type Lie superalgebras
\cite{YuanChenCao}, the super Virasoro algebras
\cite{WuGaoLiuYe}, the \(N=1\) BMS superalgebra \cite{LanLiuWu}, and the $N = 2$ Super-BMS$_3$ algebra \cite{YDL}. Local superderivations on solvable Lie and Leibniz
superalgebras were considered in~\cite{CamachoNavarroOmirov}.

In this paper, we study local derivations on the Block-type Lie algebra
\(\B\). It has a basis
\[
\{L_{\alpha,i}\mid \alpha\in\mathbb Z,\ i\in\mathbb Z_{\geq0}\}
\]
and Lie bracket
\[
[L_{\alpha,i},L_{\beta,j}]
=
\bigl(\beta(i+1)-\alpha(j+1)\bigr)
L_{\alpha+\beta,i+j},
\]
where \(\alpha,\beta\in\mathbb Z\) and
\(i,j\in\mathbb Z_{\geq0}\). 

The two indices in \(\mathcal B\) make coefficient comparisons more
complicated than those in one-index Lie algebras such as the Witt
algebra. To overcome this difficulty, we use a finite coefficient
propagation argument based on the two-index structure of \(\mathcal B\).
Since \(\mathcal B\) has nontrivial outer derivations, the proof also
requires a careful use of the structure of \(\operatorname{Der}(\mathcal
B)\). By subtracting suitable derivations and comparing coefficients,
we gradually reduce a local derivation to a derivation.

Our main result is the following.

\begin{thm}\label{thm:main}
	Every local derivation of \(\B\) is a derivation. Equivalently,
	\[
	\operatorname{LDer}(\B)=\operatorname{Der}(\B).
	\]
\end{thm}

The paper is organized as follows. Section~2 contains the definitions and
preliminary results used later, including the derivation algebra of \(\B\),
the reflection automorphism, and a coefficient elimination lemma.
Section~3 proves Theorem \ref{thm:main} through a sequence of
normalizations and coefficient comparisons.

Throughout this paper, we use the standard notation
\(\mathbb Z\), \(\mathbb Z_{\ge a}\),  \(\mathbb C\), and
\(\mathbb C^*\) for the sets of integers, integers greater than or equal to \(a\), complex numbers, and nonzero complex
numbers, respectively. All algebras are defined over \(\mathbb C\).

\section{Preliminaries}

In this section, we recall the notion of local derivations and record
several auxiliary facts used in the proof of the main theorem.

Let $\mathcal L$ be a Lie algebra. A linear map
$D:\mathcal L\to\mathcal L$ is called a derivation of $\mathcal L$ if
\[
D([x,y])
=
[D(x),y]+[x,D(y)],
\qquad
x,y\in\mathcal L.
\]
The set of all derivations of $\mathcal L$ is denoted by
$\operatorname{Der}(\mathcal L)$. For $x\in\mathcal L$, the inner
derivation $\operatorname{ad}_x$ is defined by
\[
\operatorname{ad}_x(y)=[x,y],
\qquad
y\in\mathcal L.
\]
The space of inner derivations of \(\mathcal L\) is denoted by
\(\operatorname{ad}\mathcal L\). 

\begin{de}\label{def:local-derivation}
	A linear map $\Delta:\mathcal L\to\mathcal L$ is called a
	\emph{local derivation} of $\mathcal L$ if, for every
	$x\in\mathcal L$, there exists
	$D_x\in\operatorname{Der}(\mathcal L)$ such that
	\[
	\Delta(x)=D_x(x).
	\]
	The set of all local derivations of $\mathcal L$ is denoted by
	$\operatorname{LDer}(\mathcal L)$.
\end{de}

For a Lie algebra $\mathcal L$, denote by
$\operatorname{Aut}(\mathcal L)$ its automorphism group. 

\begin{lem}\label{lem:conjugation-local-derivation}
	Let $\sigma\in\operatorname{Aut}(\mathcal L)$. If
	$\Delta\in\operatorname{LDer}(\mathcal L)$, then
	\(
	\sigma\Delta\sigma^{-1}
	\in
	\operatorname{LDer}(\mathcal L).
	\)
\end{lem}

\begin{proof}
	Fix $x\in\mathcal L$. Since $\Delta$ is a local derivation, there exists
	$D\in\operatorname{Der}(\mathcal L)$ such that
	\[
	\Delta(\sigma^{-1}(x))
	=
	D(\sigma^{-1}(x)).
	\]
	It follows that
	\[
	(\sigma\Delta\sigma^{-1})(x)
	=
	(\sigma D\sigma^{-1})(x).
	\]
	Conjugation by an automorphism preserves derivations, and hence
	\[
	\sigma D\sigma^{-1}
	\in
	\operatorname{Der}(\mathcal L).
	\]
	Therefore
	$\sigma\Delta\sigma^{-1}\in\operatorname{LDer}(\mathcal L)$.
\end{proof}

We next record a reflection automorphism of the Block-type Lie algebra
$\B$.

\begin{coro}\label{coro-B}
	Define a linear map $\sigma_0:\B\to \B$ by
	\[
	\sigma_0(L_{\alpha,i})
	=
	-L_{-\alpha,i},
	\qquad
	\alpha\in\mathbb Z,\ 
	i\in\mathbb Z_{\geq0}.
	\]
	Then $\sigma_0$ is an automorphism of $\B$. Moreover, if
	\(\Delta\in\operatorname{LDer}(\B)\), then
	$
	\sigma_0\Delta\sigma_0^{-1}\in\operatorname{LDer}(\B)$
	and 
	\[ \sigma_0\Delta\sigma_0^{-1}(L_{\alpha,i})
	=
	-\sigma_0(\Delta(L_{-\alpha,i})),\qquad
	\alpha\in\mathbb Z,\ 
	i\in\mathbb Z_{\geq0}.
	\]
	
\end{coro}
\begin{proof}
	It is straightforward to verify that \(\sigma_0\) preserves the Lie bracket. The remaining assertions follow immediately from Lemma \ref{lem:conjugation-local-derivation}.
\end{proof}

We recall the structure of the derivation
algebra of $\B$ given in \cite{XiaYouZhou}.

\begin{lem}\label{lem-der}
	The derivation algebra of $\B$ is
	\[
	\operatorname{Der}(\B)
	=
	\operatorname{ad}(\B)\oplus\mathbb C D_0,
	\]
	where $D_0$ is the outer derivation given by
	\[
	D_0(L_{\alpha,i})
	=
	iL_{\alpha,i},
	\qquad
	\alpha\in\mathbb Z,\ 
	i\in\mathbb Z_{\geq0}.
	\]
\end{lem}

The following elementary coefficient elimination lemma will be
used repeatedly in Section 3.

\begin{lem}\label{lem-coef-eli}
	Let \(a_0,a_1,\ldots,a_n\in\mathbb C\) $(n\geq 2)$. Suppose that, for every
	\(x\in\mathbb C^*\), there exist 
	\(b_0,b_1,\ldots,b_{n-1}\in\mathbb C\) such that
	\[
	a_0=\mu_0xb_0,\qquad
	a_r=\lambda_rb_{r-1}+\mu_rxb_r,\ 1\le r\le n-1,
	\qquad
	a_n=\lambda_nb_{n-1},
	\]
	where
	\(\lambda_1,\ldots,\lambda_n,\mu_0,\ldots,\mu_{n-1}\)
	are nonzero constants independent of \(x\).
	
	Then 
	\[
	a_0=a_1=\cdots=a_n=0.
	\]
\end{lem}

\begin{proof}
From \(a_0=\mu_0xb_0\), we have
$
b_0=\mu_0^{-1}x^{-1}a_0.
$
Substituting this into the equation for \(a_1\), and continuing
inductively, we obtain
\[
b_r=\mu_r^{-1}x^{-1}a_r+
\sum_{t=1}^{r}\xi_{r,t}x^{-t-1}a_{r-t},
\]
where all \(\xi_{r,t}\) are nonzero constants independent of \(x\).
Substitution into \(a_n=\lambda_nb_{n-1}\) gives
\[
a_n+\eta_1x^{-1}a_{n-1}
+\cdots+\eta_nx^{-n}a_0=0,
\qquad x\in\mathbb C^*,
\]
where	\(\eta_1,\ldots,\eta_n\in\mathbb C\) are nonzero constants independent of \(x\). After multiplying the resulting relation by a
	suitable power of \(x\), we obtain a polynomial identity in \(x\).
	Since the identity holds for infinitely many values of \(x\), all its
	coefficients vanish, which gives
	\(a_0=a_1=\cdots=a_n=0\).
\end{proof}

\section{Local Derivations on \(\B\)}

In this section, we prove Theorem~1.1 by gradually simplifying a local
derivation. Let \(\Delta\in\operatorname{LDer}(\mathcal B)\). Using the
description of \(\operatorname{Der}(\mathcal B)\), we subtract suitable
derivations step by step and study the resulting action of \(\Delta\) on
the basis elements of \(\mathcal B\). The main ingredients are
Lemma~3.6, which allows us to normalize the values of \(\Delta\) on
\(L_{0,i}\) and \(L_{\alpha,0}\), and Lemma~3.9, which shows that the
normalized local derivation is zero on all elements of \(\mathcal B\).

Throughout this section, all sums are finite. Unless otherwise stated,
we omit zero coefficients and assume that the extreme terms in each
finite sum have nonzero coefficients.

We first consider the action of \(\Delta\) on the elements
\(L_{0,i}\).

\begin{lem}\label{lem1}
	Let \(\Delta\) be a local derivation of \(\B\).
	Replacing \(\Delta\) by \(\Delta-D\) for a suitable
	\(D\in\operatorname{Der}(\B)\) if necessary, we may assume that
	\[
	\Delta(L_{0,i})=0,
	\qquad
	i\in\mathbb Z_{\ge0}.
	\]
\end{lem}

\begin{proof}
	By locality, there exists \(D'\in\operatorname{Der}(\B)\) such that
	\(\Delta(L_{0,0})=D'(L_{0,0})\). Since the difference of a local
	derivation and a derivation is again a local derivation, replacing
	\(\Delta\) by \(\Delta-D'\), and retaining the notation \(\Delta\), we
	may assume that
	\begin{equation}\label{eq:vertical-L00}
		\Delta(L_{0,0})=0.
	\end{equation}
	
	Fix \(i\in\mathbb Z_{\geq1}\). Grouping the second indices modulo \(i\),
	we write
	\begin{equation}\label{eq:vertical-initial-expansion}
		\Delta(L_{0,i})
		=
		\sum_{\gamma\in\Gamma}
		\sum_{j\in A_\gamma}
		\sum_{s=p_{\gamma,j}}^{q_{\gamma,j}}
		c_{\gamma,j+si}L_{\gamma,j+si},
	\end{equation}
	where \(\Gamma\) is a finite subset of \(\mathbb Z\),
	\(A_\gamma\subseteq\{0,\ldots,i-1\}\), $c_{\gamma,j+si}\in\C$, and
	\(0\le p_{\gamma,j}\le q_{\gamma,j}\) are integers. 
	
	For \(x\in\mathbb C^*\), set \(u=L_{0,i}+xL_{0,0}\). By
	\eqref{eq:vertical-L00} and linearity,
	\(\Delta(u)=\Delta(L_{0,i})\). By locality and
	Lemma~\ref{lem-der}, there exist \(v\in \B\) and
	\(d\in\mathbb C\), possibly depending on \(x\), such that
	\[
	\Delta(u)=[v,u]+dD_0(u).
	\]
Write
	\[
	v
	=
	\sum_{\gamma\in\Gamma'}
	\sum_{j\in A'_\gamma}
	\sum_{s=p'_{\gamma,j}}^{q'_{\gamma,j}}
	b_{\gamma,j+si}L_{\gamma,j+si},
	\]
	where \(\Gamma'\) is a finite subset of \(\mathbb Z\),
	\(A'_\gamma\subseteq\{0,\ldots,i-1\}\), $b_{\gamma,j+si}\in\C$, and
	\(0\leq p'_{\gamma,j}\leq q'_{\gamma,j}\) are integers. Then
	\begin{align}
		\Delta(L_{0,i})
		&=
		-\sum_{\gamma\in\Gamma'}
		\sum_{j\in A'_\gamma}
		\sum_{s=p'_{\gamma,j}}^{q'_{\gamma,j}}
		(i+1)\gamma b_{\gamma,j+si}
		L_{\gamma,j+(s+1)i}
		\nonumber\\
		&\quad
		-x\sum_{\gamma\in\Gamma'}
		\sum_{j\in A'_\gamma}
		\sum_{s=p'_{\gamma,j}}^{q'_{\gamma,j}}
		\gamma b_{\gamma,j+si}L_{\gamma,j+si}
		+diL_{0,i}.
		\label{eq:vertical-local-expansion}
	\end{align}
	
	Fix
	\(0\neq \gamma\in\Gamma\) and
	\(j\in A_\gamma\).
	Comparing the coefficient of
	\(L_{\gamma,j+si}\) in
	\eqref{eq:vertical-initial-expansion} and
	\eqref{eq:vertical-local-expansion}, we obtain
	\begin{equation}\label{eq:vertical-recurrence}
		c_{\gamma,j+si}
		=
		-(i+1)\gamma b_{\gamma,j+(s-1)i}
		-x\gamma b_{\gamma,j+si}.
	\end{equation}
	Since \(c_{\gamma,j+si}=0\) for \(s<p_{\gamma,j}\) and
	\(s>q_{\gamma,j}\), and only finitely many coefficients
	\(b_{\gamma,j+si}\) are nonzero, \eqref{eq:vertical-recurrence} implies
	\[
	b_{\gamma,j+si}=0,
	\qquad s<p_{\gamma,j}\ \text{or}\  s\ge q_{\gamma,j}.
	\]
	If \(p_{\gamma,j}=q_{\gamma,j}\), then \eqref{eq:vertical-recurrence} gives
	\(c_{\gamma,j+p_{\gamma,j}i}=0\), contrary to the choice of the coefficient string.
	Hence \(p_{\gamma,j}<q_{\gamma,j}\), and the relations for \(s=p_{\gamma,j},\ldots,q_{\gamma,j}\) in \eqref{eq:vertical-recurrence} reduce to
	\[
	\begin{aligned}
		c_{\gamma,j+p_{\gamma,j}i}
		&=-x\gamma b_{\gamma,j+p_{\gamma,j}i},\\
		c_{\gamma,j+(s+1)i}
		&=-(i+1)\gamma b_{\gamma,j+si}
		-x\gamma b_{\gamma,j+(s+1)i},
		\qquad p_{\gamma,j}\leq s\leq q_{\gamma,j}-2,\\
		c_{\gamma,j+q_{\gamma,j}i}
		&=-(i+1)\gamma b_{\gamma,j+(q_{\gamma,j}-1)i}.
	\end{aligned}
	\]
	Since \(\gamma\neq0\), applying Lemma ~\ref{lem-coef-eli}, we obtain
	$$
	c_{\gamma,j+si}=0,\qquad p_{\gamma,j}\leq s\leq q_{\gamma,j}.$$
	
	For \(\gamma=0\), comparison of the coefficients in \eqref{eq:vertical-initial-expansion} and
	\eqref{eq:vertical-local-expansion} shows that $c_{0,k}=0$  for $k\neq i$. From \eqref{eq:vertical-initial-expansion}, we have
	$$\Delta(L_{0,i})=c_iL_{0,i},\qquad i\in\mathbb Z_{\geq0},$$
	where
	\(c_i\in\mathbb C\).
	 Moreover, \(c_0=0\) by \eqref{eq:vertical-L00}.
	
	Set
	\[
	\Delta'=\Delta-c_1D_0.
	\]
	Then \(\Delta'\) is a local derivation,
	\(\Delta'(L_{0,0})=\Delta'(L_{0,1})=0\), and
	\[
	\Delta'(L_{0,i})
	=
	(c_i-ic_1)L_{0,i},
	\qquad
	i\in\Z_{\geq2}.
	\]
	
	Fix \(i\geq2\) and set \(w=L_{0,i}+L_{0,1}\). Thus
	$
		\Delta'(w)
		=
		(c_i-ic_1)L_{0,i}.
	$
	By locality and Lemma~\ref{lem-der}, there exist
	\[
	z=\sum_{(\gamma,k)\in S}
	d_{\gamma,k}L_{\gamma,k}\in \B,
	\qquad
	d'\in\mathbb C,
	\]
	where \(S\subseteq\mathbb Z\times\mathbb Z_{\geq0}\) is finite, $d_{\gamma,k}\in\C$, such that
	\(
	\Delta'(w)=[z,w]+d'D_0(w).
	\)
	Consequently,
	\begin{align}
		(c_i-ic_1)L_{0,i}
		=
		-\sum_{(\gamma,k)\in S}
		(i+1)\gamma d_{\gamma,k}L_{\gamma,k+i}
		-2\sum_{(\gamma,k)\in S}
		\gamma d_{\gamma,k}L_{\gamma,k+1}
		+d'\bigl(iL_{0,i}+L_{0,1}\bigr).
		\label{eq:second-vertical-local}
	\end{align}
	Comparing the coefficient of \(L_{0,1}\) on both sides of
	\eqref{eq:second-vertical-local} gives \(d'=0\). Comparing the coefficient
	of \(L_{0,i}\) then gives
	$
	c_i-ic_1=0.
	$
	Therefore,
	\[
	\Delta'(L_{0,i})=0,
	\qquad
	i\in\mathbb Z_{\geq0}.
	\]
	
	This proves the lemma.
\end{proof}

After the preceding normalization, we turn to the elements
\(L_{\alpha,0}\). The next five lemmas successively show that, after
subtracting suitable derivations, the local derivation vanishes on all
\(L_{\alpha,0}\).

\begin{lem}\label{lem2}
	Let \(\Delta\) be a local derivation of \(\B\) satisfying
	$
	\Delta(L_{0,0})=0.
	$
	Then, for every \(\alpha\in\mathbb Z\), there exist a finite subset
	\(K\subseteq\mathbb Z_{\geq0}\) and coefficients
	\(c_{s\alpha,k}\in\mathbb C\) such that
	\[
	\Delta(L_{\alpha,0})
	=
	\sum_{k\in K}\sum_{s=1}^{k+1}
	c_{s\alpha,k}L_{s\alpha,k}.
	\]
\end{lem}

\begin{proof}
	The assertion is immediate for \(\alpha=0\).
	We first assume that \(\alpha \in\mathbb Z_{\geq1}\).
	
	Grouping the first indices modulo \(\alpha\), we may write
	\begin{equation}\label{lem3.2-1}
		\Delta(L_{\alpha,0})
		=
		\sum_{k\in K}
		\sum_{j\in B_k}
		\sum_{s=l_{k,j}}^{m_{k,j}}
		c_{j+s\alpha,k}L_{j+s\alpha,k},
	\end{equation}
	where \(K\subseteq\mathbb Z_{\geq0}\) is finite,
	\(B_k\subseteq\{0,\ldots,\alpha-1\}\), $c_{j+s\alpha,k}\in\C$, and
	\(l_{k,j}\leq m_{k,j}\) are integers.
	
	For \(x\in\mathbb C^*\), set \(u=L_{\alpha,0}+xL_{0,0}\). Since
	\(\Delta(L_{0,0})=0\), linearity gives \(\Delta(u)=\Delta(L_{\alpha,0})\).
	By locality and Lemma~\ref{lem-der}, there exist
	\(v\in \B\) and \(d\in\mathbb C\), possibly depending on \(x\), such that
	\[
	\Delta(u)=[v,u]+dD_0(u).
	\]
	Write
	\[
	v
	=
	\sum_{k,j,s}
	b_{j+s\alpha,k}L_{j+s\alpha,k},
	\]
	where the sum is finite and $b_{j+s\alpha,k}\in\C$. Using the defining bracket of \(\B\), we obtain
	\begin{align}
		\Delta(L_{\alpha,0})
		=
		\sum_{k,j,s}
		\bigl((k+1-s)\alpha-j\bigr)
		b_{j+s\alpha,k}L_{j+(s+1)\alpha,k}
		-x\sum_{k,j,s}
		(j+s\alpha)b_{j+s\alpha,k}L_{j+s\alpha,k}.
		\label{lem3.2-2}
	\end{align}
	Comparing the coefficient of
	\(L_{j+s\alpha,k}\) in
	\eqref{lem3.2-1} and
	\eqref{lem3.2-2}, we obtain
	\begin{equation}\label{lem3.2-3}
		c_{j+s\alpha,k}
		=
		\bigl((k+2-s)\alpha-j\bigr)b_{j+(s-1)\alpha,k}
		-x(j+s\alpha)b_{j+s\alpha,k}.
	\end{equation}
	
	\medskip
	\noindent
	\textbf{Claim 1.} For every $k\in K$,  $B_k=\{0\}$.
	
	If \(\alpha=1\), the claim is
	immediate.  Assume that $\alpha\geq2$.
	Fix $k\in K$ and suppose that
	$j\in B_k$ with
	$1\leq j\leq\alpha-1$.
	By the choice of \(l_{k,j}\) and \(m_{k,j}\),
	\(c_{j+s\alpha,k}=0\) whenever \(s<l_{k,j}\) or \(s>m_{k,j}\), whereas
	$
	c_{j+l_{k,j}\alpha,k}\neq0$ and $c_{j+m_{k,j}\alpha,k}\neq0$.
	Note that both \(j+s\alpha\) and \((k+2-s)\alpha-j\) are nonzero for
	every integer \(s\). Together with \eqref{lem3.2-3} and the fact that only finitely
	many coefficients \(b_{j+s\alpha,k}\) are nonzero, we conclude that
	\[
	b_{j+s\alpha,k}=0,
	\qquad
	s<l_{k,j}\ \text{or}\ s\ge m_{k,j}.
	\]

	If $l_{k,j}=m_{k,j}$, then
	$
	c_{j+l_{k,j}\alpha,k}
	=
	-x(j+l_{k,j}\alpha)b_{j+l_{k,j}\alpha,k}=0,
	$
	which contradicts the choice of $l_{k,j}$.
	Hence $l_{k,j}<m_{k,j}$. From
	\eqref{lem3.2-3}, we have
	\[
	\begin{aligned}
		c_{j+l_{k,j}\alpha,k}
		&=
		-x(j+l_{k,j}\alpha)b_{j+l_{k,j}\alpha,k},\\
		c_{j+s\alpha,k}
		&=
		((k+2-s)\alpha-j)b_{j+(s-1)\alpha,k}
		-x(j+s\alpha)b_{j+s\alpha,k},
		\quad l_{k,j}<s<m_{k,j},
		\\
		c_{j+m_{k,j}\alpha,k}
		&=
		((k+2-m_{k,j})\alpha-j)
		b_{j+(m_{k,j}-1)\alpha,k}.
	\end{aligned}
	\]
	Since \(1\leq j\leq\alpha-1\), Lemma \ref{lem-coef-eli}
	gives
	\[
	c_{j+l_{k,j}\alpha,k}
	=
	\cdots
	=
	c_{j+m_{k,j}\alpha,k}
	=
	0,
	\]
	contrary to the choice of the coefficient string. Therefore
	$B_k=\{0\}$.
	
	By Claim 1, \eqref{lem3.2-1} becomes
	\begin{equation*}
		\Delta(L_{\alpha,0})
		=
		\sum_{k\in K}\sum_{s=l_k}^{m_k}
		c_{s\alpha,k}L_{s\alpha,k},
	\end{equation*}	
	where \(l_k\leq m_k\) are the smallest and largest indices for which
	\(c_{s\alpha,k}\neq0\).
	
	\medskip
	\noindent
	\textbf{Claim 2.}
	For every \(k\in K\),
	$
	l_k\geq1.
	$
	
	Fix $k\in K$ and suppose that $l_k\leq0$. 
	Putting $j=0$ in
	\eqref{lem3.2-3}, we obtain
	\begin{equation}\label{lem3.2-5}
		c_{s\alpha,k}
		=
		(k+2-s)\alpha b_{(s-1)\alpha,k}
		-
		xs\alpha b_{s\alpha,k}.
	\end{equation}
Since
	$c_{s\alpha,k}=0$ for $s<l_k$, relation
	\eqref{lem3.2-5} implies
	$
	b_{s\alpha,k}=0$ for $s<l_k$.
	
	If $l_k=0$,  then \(b_{-\alpha,k}=0\), and hence
	$
	c_{0,k}
	=
	(k+2)\alpha b_{-\alpha,k}=0,
	$
	contradicting the definition of $l_k$.

If $l_k<0$, from \eqref{lem3.2-5}, we have
\[
\begin{aligned}
	c_{l_k\alpha,k}
	&=
	-xl_k\alpha b_{l\alpha,k},\\
	c_{s\alpha,k}
	&=
	(k+2-s)\alpha b_{(s-1)\alpha,k}
	-xs\alpha b_{s\alpha,k},
	\quad l_k<s<0,
	\\
	c_{0,k}
	&=
	(k+2)\alpha b_{-\alpha,k}.
\end{aligned}
\]
All transition coefficients are nonzero. Applying   Lemma \ref{lem-coef-eli}, we obtain
\(c_{l_k\alpha,k}=\cdots=c_{0,k}=0\), contradicting
\(c_{l_k\alpha,k}\neq0\). Therefore \(l_k\geq1\), proving Claim~2.

\medskip
\noindent
\textbf{Claim 3.}
For every $k\in K$,
$
m_k\leq k+1.
$

Fix $k\in K$ and suppose that
$m_k\geq k+2$. Since
$c_{s\alpha,k}=0$ for $s>m_k$, relation
\eqref{lem3.2-5} implies
$
b_{s\alpha,k}=0$ for $ s\geq m_k$.

If $m_k=k+2$, then \(b_{m_k\alpha,k}=0\), and
\eqref{lem3.2-5} gives
$
c_{m_k\alpha,k}=0,
$
contradicting the definition of $m_k$.

Assume that \(m_k>k+2\). Considering
\eqref{lem3.2-5} for
$k+2\leq s\leq m_k$, we obtain
\[
\begin{aligned}
	c_{(k+2)\alpha,k}
	&=-x(k+2)\alpha\,b_{(k+2)\alpha,k},\\
	c_{s\alpha,k}
	&=(k+2-s)\alpha\,b_{(s-1)\alpha,k}
	-xs\alpha\,b_{s\alpha,k},
	\qquad k+2<s<m_k,\\
	c_{m_k\alpha,k}
	&=(k+2-m_k)\alpha\,b_{(m_k-1)\alpha,k}.
\end{aligned}
\]
All transition coefficients are nonzero. Hence
 Lemma \ref{lem-coef-eli} yields
\(c_{(k+2)\alpha,k}=\cdots=c_{m_k\alpha,k}=0\), contradicting
\(c_{m_k\alpha,k}\neq0\). Thus \(m_k\leq k+1\), proving Claim~3.

Combining Claims~1--3, we obtain
\[
\Delta(L_{\alpha,0})
=
\sum_{k\in K}\sum_{s=1}^{k+1}
c_{s\alpha,k}L_{s\alpha,k},
\qquad
\alpha\in\mathbb Z_{\geq1}.
\]

It remains to consider \(\alpha\in\mathbb Z_{\leq-1}\). Put
\(\beta=-\alpha\). By
Corollary~\ref{coro-B}, the map
$
\widetilde{\Delta}
=
\sigma_0\Delta\sigma_0^{-1}
$
is a local derivation satisfying
\(\widetilde{\Delta}(L_{0,0})=0\). Applying the positive case to
\(\widetilde{\Delta}(L_{\beta,0})\), we obtain
\[
\widetilde{\Delta}(L_{\beta,0})
=
\sum_{k\in F}\sum_{s=1}^{k+1}
d_{s\beta,k}L_{s\beta,k}
\]
for some finite set \(F\subseteq\mathbb Z_{\geq0}\) and $d_{s\beta,k}\in\C$. Since
$
\widetilde{\Delta}(L_{\beta,0})
=
-\sigma_0\bigl(\Delta(L_{\alpha,0})\bigr),
$
conjugating back by \(\sigma_0\) gives
\[
\Delta(L_{\alpha,0})
=
\sum_{k\in F}\sum_{s=1}^{k+1}
d_{s\beta,k}L_{s\alpha,k}.
\]
Renaming the coefficients completes the proof.
\end{proof}

The next lemma shows that only the terms \(L_{\alpha,k}\) can occur in
\(\Delta(L_{\alpha,0})\).

\bl{lem3}
	Let $\Delta$ be a local derivation of $\B$ satisfying
	$\Delta(L_{0,0})=0$. Then, for every $\alpha\in\Z$, there exist a finite
	subset $K\subseteq\Z_{\geq0}$ and coefficients $c_{\alpha,k}\in\C$ such that
	\[
	\Delta(L_{\alpha,0})
	=\sum_{k\in K}c_{\alpha,k}L_{\alpha,k}.
	\]
\el

\begin{proof}
	The assertion is immediate for $\alpha=0$. Assume first that $\alpha\in\Z_{\geq1}$.
	By Lemma \ref{lem2},  we may
	write
	\begin{equation*}\label{eq:Delta-alpha}
		\Delta(L_{\alpha,0})
		=\sum_{k=p}^q\sum_{s=1}^{k+1}
		c_{s\alpha,k}L_{s\alpha,k}
	\end{equation*}
	for some integers $p\leq q,0\leq k\in\mathbb Z$, and $c_{s\alpha,k}\in\C$. Choose an integer
	$\beta>(q+2)\alpha$. By Lemma \ref{lem2}, there exist a finite
	subset $K'\subseteq\Z_{\geq0}$ and coefficients $d_{s\beta,k}\in\C$ such that
	\begin{equation*}\label{eq:Delta-beta}
		\Delta(L_{\beta,0})
		=\sum_{k\in K'}\sum_{s=1}^{k+1}
		d_{s\beta,k}L_{s\beta,k}.
	\end{equation*}
	
	For $x\in\C^{*}$, set
	$u=L_{\alpha,0}+xL_{\beta,0}$. By linearity, we have
	\begin{equation}\label{eq:Delta-w-linear}
		\Delta(u)
		=\sum_{k=p}^q\sum_{s=1}^{k+1}
		c_{s\alpha,k}L_{s\alpha,k}
		+x\sum_{k\in K'}\sum_{s=1}^{k+1}
		d_{s\beta,k}L_{s\beta,k}.
	\end{equation}
	By locality and the description of $\Der(\B)$, there exist $v\in \B$ and
	$d\in\C$, possibly depending on $x$, such that
	$\Delta(u)=[v,u]+dD_0(u)$. Since $D_0(u)=0$, writing
	$v=\sum_{(\gamma,k)\in S}b_{\gamma,k}L_{\gamma,k}$ gives
	\begin{equation}\label{eq:Delta-w-local}
			\Delta(u)
			=\sum_{(\gamma,k)\in S}
			\bigl((k+1)\alpha-\gamma\bigr)b_{\gamma,k}
			L_{\gamma+\alpha,k} 
			+x\sum_{(\gamma,k)\in S}
			\bigl((k+1)\beta-\gamma\bigr)b_{\gamma,k}
			L_{\gamma+\beta,k},
	\end{equation}
	where \(S\subseteq\mathbb Z\times\mathbb Z_{\geq0}\) is finite and $b_{\gamma,k}\in\C$.
	
	Fix $k\in \{p,\dots,q\}$. 
	
	\medskip
	\noindent\textbf{Claim 1.}
	If $b_{\gamma,k}\neq0$, then $\gamma\geq0$.
	
	Suppose that $b_{\gamma,k}\neq0$ for some $\gamma<0$. 
	For every
	\(t\geq0\), compare the coefficient of
	$
	L_{\gamma+(t+1)\alpha-t\beta,k}
	$
	in \eqref{eq:Delta-w-linear} and
	\eqref{eq:Delta-w-local}. The basis vector does not occur in \eqref{eq:Delta-w-linear} because 
	$\gamma+(t+1)\alpha-t\beta<\alpha$. Hence we obtain
	\begin{align*}
		0
		=\bigl((k+1-t)\alpha+t\beta-\gamma\bigr)
		b_{\gamma+t(\alpha-\beta),k} 
		+x\bigl((k+t+2)\beta-(t+1)\alpha-\gamma\bigr)
		b_{\gamma+(t+1)(\alpha-\beta),k}.
	\end{align*}
	Since $\beta>\alpha$ and $\gamma<0$, we have $(k+1-t)\alpha+t\beta-\gamma=(k+1)\alpha+t(\beta-\alpha)-\gamma>0$ and $(k+t+2)\beta-(t+1)\alpha-\gamma \neq0$.  Hence
	$b_{\gamma+t(\alpha-\beta),k}\neq0$ implies
	$b_{\gamma+(t+1)(\alpha-\beta),k}\neq0$. Starting with
	$b_{\gamma,k}\neq0$, this produces infinitely many nonzero coefficients
	of $v$, a contradiction. This proves Claim~1.
	
	\medskip
	\noindent
	\textbf{Claim 2.}
	If
	\(b_{\gamma,k}\neq0\), then $
	\gamma=r\beta+m\alpha 
	$
	for some \(r,m\in\mathbb Z_{\geq0}\) satisfying 
	\(
	0\leq r\leq k+1\) and \(0\leq m\leq k+1-r\).

	By Claim~1, we have \(\gamma\geq0\). Suppose, to the contrary, that
	\(b_{\gamma,k}\neq0\), but that
	$\gamma$ has no representation of the stated form. We prove inductively that
	\begin{equation}\label{eq:claim2-propagation}
		b_{\gamma+t(\beta-\alpha),k}\neq0,
		\qquad t\in\mathbb Z_{\geq0}.
	\end{equation}
	The assertion is true for \(t=0\) by assumption.
	
	Assume that
	\(b_{\gamma+t(\beta-\alpha),k}\neq0\) for some \(t\geq0\), and compare
	the coefficient of \(L_{\gamma+(t+1)\beta-t\alpha,k}\)
	in \eqref{eq:Delta-w-linear} and \eqref{eq:Delta-w-local}. We first show
	that this basis vector does not occur in
	\eqref{eq:Delta-w-linear}.

Indeed, suppose first that
$
\gamma+(t+1)\beta-t\alpha=\ell\alpha
$
for some \(1\leq\ell\leq k+1\). Then
\[
\gamma
=(t+\ell)\alpha-(t+1)\beta
=(\ell-1)\alpha-(t+1)(\beta-\alpha)
\leq k\alpha-(\beta-\alpha)<0,
\]
because \(\ell\leq k+1\leq q+1\) and
\(\beta>(q+2)\alpha\). This contradicts Claim~1. 

Suppose next that
$
\gamma+(t+1)\beta-t\alpha=\ell\beta
$
for some \(1\leq\ell\leq k+1\). If \(\ell\geq t+1\), then
$
\gamma=(\ell-t-1)\beta+t\alpha
$
is of the form \(r\beta+m\alpha\), where
\(r,m\in\mathbb Z_{\geq0}\) and
\(r+m=\ell-1\leq k\), contrary to the assumption on \(\gamma\).
Hence \(\ell\leq t\), and therefore
\[
\gamma
=(\ell-t-1)\beta+t\alpha=(\ell-1)\alpha-(t+1-\ell)(\beta-\alpha)
\leq k\alpha-(\beta-\alpha)<0,
\]
again contradicting Claim~1.

	Thus the coefficient under comparison is zero on the left, while on the
right we have
\begin{align*}
	0
	=\bigl((k+t+2)\alpha-(t+1)\beta-\gamma\bigr)
	b_{\gamma+(t+1)(\beta-\alpha),k}
	+x\bigl((k+1-t)\beta+t\alpha-\gamma\bigr)
	b_{\gamma+t(\beta-\alpha),k}.
\end{align*}
Consequently,
$b_{\gamma+t(\beta-\alpha),k}\neq0$ implies
$b_{\gamma+(t+1)(\beta-\alpha),k}\neq0$. This produces infinitely
many nonzero coefficients of $v$, a contradiction. Claim~2 follows.	

\medskip
	\noindent
	\textbf{Claim 3.} $c_{s\alpha,k}=0$ for \(2\leq s\leq k+1\).
    
 Claims 1 and 2 imply $b_{s\alpha-\beta,k}=b_{s\beta-\alpha,k}=0$. Comparing the coefficients of
$L_{r\beta+(s-r)\alpha,k}$ for $0\leq r\leq s$ in \eqref{eq:Delta-w-linear} and \eqref{eq:Delta-w-local}  gives
\begin{equation*}
	\begin{aligned}
		c_{s\alpha,k}&=(k+2-s)\alpha b_{(s-1)\alpha,k},\\
		0&=((k+2-s+r)\alpha-r\beta)b_{r\beta+(s-r-1)\alpha,k}+x((k+2-r)\beta-(s-r)\alpha)b_{(r-1)\beta+(s-r)\alpha,k},\\
        d_{s\beta,k}&=(k+2-s)\beta b_{(s-1)\beta,k},
	\end{aligned}
\end{equation*}
where $1\leq r\leq s-1$. Since \(\beta>(q+2)\alpha\) and $k\leq q$, we have $(k+3-s)\alpha-\beta>0$  and $(k+2-r)\beta-(s-r)\alpha>0$. 
Since \(c_{s\alpha,k}\) and \(d_{s\beta,k}\) are independent of \(x\),
whereas the above identity holds for every \(x\in\mathbb C^*\), we must have
\[
c_{s\alpha,k}=d_{s\beta,k}=0, \qquad
2\leq s\leq k+1.
\]
This proves Claim~3.

Thus 
\[
\Delta(L_{\alpha,0})
=\sum_{k=p}^qc_{\alpha,k}L_{\alpha,k}
\qquad \alpha\in\mathbb Z_{\geq1}.
\]

Finally, for \(\alpha\in\mathbb Z_{\leq-1}\), let
\[
\widetilde{\Delta}=\sigma_0\Delta\sigma_0^{-1}.
\]
By Corollary~\ref{coro-B}, \(\widetilde{\Delta}\) is a local derivation satisfying
the same assumptions. Applying the positive case to
\(\widetilde{\Delta}(L_{-\alpha,0})\) and conjugating back by
\(\sigma_0\), we obtain the desired conclusion for
\(\alpha\leq-1\).

The proof is complete.
\end{proof}


\bl{lem4}
	Let $\Delta$ be a local derivation of $\B$ satisfying
	$\Delta(L_{0,0})=0$. Replacing \(\Delta\) by \(\Delta-D\) for a suitable
	\(D\in\operatorname{Der}(\B)\) if necessary, we may assume that
	\[
	\Delta(L_{0,i})=\Delta(L_{\alpha,0})=0,
	\quad i\in\mathbb Z_{\geq0},
	\ \alpha\in\mathbb Z_{\geq2}.
	\]
\el
\begin{proof}
	By Lemma~\ref{lem1}, after replacing  \(\Delta\) by
	$\Delta-D$ for a suitable $D\in\operatorname{Der}(\B)$,
	we may assume that
	$\Delta(L_{0,i})=0$ for all
	$i\in\mathbb Z_{\geq0}$. 
	
	Fix $\alpha\in\mathbb Z_{\geq2}$ and put $\beta=\alpha+1$. From Lemma \ref{lem3}, we
	may choose a finite subset $K\subseteq\mathbb Z_{\geq0}$, $c_{\alpha,k},c_{\beta,k}\in\C$, such that
	\begin{equation*}
		\Delta(L_{\alpha,0})
		=
		\sum_{k\in K}c_{\alpha,k}L_{\alpha,k},
		\qquad
		\Delta(L_{\beta,0})
		=
		\sum_{k\in K}c_{\beta,k}L_{\beta,k}.
	\end{equation*}
	Choose an integer $i\geq1$ larger than every element of $K$,  and set
	$u=L_{\alpha,0}+L_{\beta,0}+L_{0,i}$. Since
	\(\Delta(L_{0,i})=0\),
	linearity gives
	\begin{equation}\label{eq:adjacent-linear}
		\Delta(u)
		=
		\sum_{k\in K}c_{\alpha,k}L_{\alpha,k}
		+
		\sum_{k\in K}c_{\beta,k}L_{\beta,k}.
	\end{equation}
By locality and Lemma~\ref{lem-der}, there exist
	\[
	v=\sum_{(\gamma,k)\in S}b_{\gamma,k}L_{\gamma,k}\in \B,
	\qquad d\in\mathbb C,
	\]
such that
	\[
	\Delta(u)=[v,u]+dD_0(u),
	\] 
	where $S\subseteq\mathbb Z\times\mathbb Z_{\geq0}$ is finite and $b_{\gamma,k}\in\C$. Using the defining bracket of $\B$, we have 
	\begin{equation}\begin{aligned}
		\Delta(u)
		&=
		\sum_{(\gamma,k)\in S}
		\bigl((k+1)\alpha-\gamma\bigr)
		b_{\gamma,k}L_{\gamma+\alpha,k}
		+
		\sum_{(\gamma,k)\in S}
		\bigl((k+1)\beta-\gamma\bigr)
		b_{\gamma,k}L_{\gamma+\beta,k}
		\\
		&\quad-
		\sum_{(\gamma,k)\in S}
		(i+1)\gamma b_{\gamma,k}L_{\gamma,k+i}
		+
		diL_{0,i}.
		\label{eq:adjacent-local}
	\end{aligned}
	\end{equation}
Comparing the coefficients of
\(L_{\alpha,k}\),
\(L_{\beta,k}\), and \(L_{1,k+i}\) in
\eqref{eq:adjacent-linear} and \eqref{eq:adjacent-local} gives

\begin{align}
	c_{\alpha,k}&=(k+1)\alpha b_{0,k}+((k+1)(\alpha+1)+1)b_{-1,k},\label{3.17'}\\
	c_{\beta,k}
	&=
	(k+1)\beta b_{0,k}
	+
	\bigl((k+1)\alpha-1\bigr)b_{1,k},\label{3.18}\\
	0&=\bigl((k+i+2)\alpha-1\bigr)b_{1-\alpha,k+i}+
	\bigl((k+i+2)\beta-1\bigr)b_{1-\beta,k+i}
	-(i+1)b_{1,k}.\label{3.17}
\end{align}
	
We first show that
\begin{equation}\label{3.19}
b_{\gamma,k}=0
\qquad\text{whenever }\gamma<0.
\end{equation}
Suppose that $b_{\gamma,k}\neq0$ for some $\gamma<0$. Since
$k+i$ is larger than every element of $K$, the basis vector $L_{\gamma,k+i}$ does not occur in
\eqref{eq:adjacent-linear}. Comparing its coefficient in
\eqref{eq:adjacent-linear} and \eqref{eq:adjacent-local}, we obtain
\[
0=
\bigl((k+i+2)\alpha-\gamma\bigr)b_{\gamma-\alpha,k+i}
+
\bigl((k+i+2)\beta-\gamma\bigr)b_{\gamma-\beta,k+i}
-
(i+1)\gamma b_{\gamma,k}.
\]
Since $(k+i+2)\alpha-\gamma\neq0$, $(k+i+2)\beta-\gamma\neq0$ and $-(i+1)\gamma\neq0$, then at least one of 
$b_{\gamma-\alpha,k+i}$ and
$b_{\gamma-\beta,k+i}$ is nonzero. Repeating this argument produces
infinitely many nonzero coefficients of $v$, with strictly increasing
second indices, contradicting the finiteness of $v$. This proves (\ref{3.19}).

Hence
$b_{-1,k}=b_{1-\alpha,k+i}=b_{1-\beta,k+i}=0$ in (\ref{3.17'}) and (\ref{3.17}). It follows  that
$b_{1,k}=0$ and $c_{\alpha,k}=(k+1)\alpha b_{0,k}$.  Substituting this into (\ref{3.18}), we obtain
	\(
	c_{\beta,k}=(k+1)\beta b_{0,k},
	\)
	and hence
	\begin{equation}\label{eq:adjacent-ratio}
		c_{\alpha,k}
		=
		\frac{\alpha}{\beta}c_{\beta,k}.
	\end{equation}
Let $c_k=\frac12c_{2,k}$. Since $\beta=\alpha+1$,
it follows by induction that
	\(
	c_{\alpha,k}
	=\alpha
	c_{k}
	\)
for all \(\alpha\in\mathbb Z_{\geq2}\). Therefore,
	\begin{equation}
		\Delta(L_{\alpha,0})
		=
		\alpha\sum_{k\in K}c_kL_{\alpha,k}.
	\end{equation}
Define
\[
D'
=
\sum_{k\in K}
\frac{c_k}{k+1}
\operatorname{ad}L_{0,k}.
\]
Then $D'\in\operatorname{Der}(\B)$ and  $\Delta-D'\in\operatorname{LDer}(\B)$. Using the defining bracket of $\B$, we have
\[
(\Delta-D')(L_{0,i})=0,
\quad i\in\mathbb Z_{\geq0} ,
\qquad
(\Delta-D')(L_{\alpha,0})=0
\quad \alpha\in\mathbb Z_{\geq2} .
\]	
The proof is complete.
\end{proof}

\bl{lem5}
	Let $\Delta$ be a local derivation of $\B$. Replacing
	$\Delta$ by $\Delta-D$ for a suitable
	$D\in\operatorname{Der}(\B)$ if necessary, we may assume that
	\[
	\Delta(L_{0,i})=\Delta(L_{\alpha,0})=0,
	\qquad
	i\in\mathbb Z_{\ge0},\ 
	\alpha\in\mathbb Z_{\ge1}.
	\]
\el

\begin{proof}
	By Lemmas~\ref{lem1} and \ref{lem4}, after replacing
	\(\Delta\)
	by
	\(\Delta-D\)
	for a suitable derivation if necessary, we may assume
	\[
	\Delta(L_{0,i})=\Delta(L_{\alpha,0})=0,
	\qquad
	i\in\mathbb Z_{\ge0},\ 
	\alpha\in\mathbb Z_{\ge2}.
	\]
	It remains to prove that
	\[
	\Delta(L_{1,0})=0.
	\]

	By Lemma~\ref{lem3}, we suppose 
	\begin{equation}\label{eq:Delta-L10}
		\Delta(L_{1,0})
		=
		\sum_{k=p}^{q}c_{1,k}L_{1,k}
	\end{equation}
	for some integers \(0\leq p\leq q\) and $c_{1,k}\in\C$. Choose an integer \(\alpha\geq2\) such that $\alpha-1>q+1$, and set
	$u=L_{1,0}+L_{\alpha,0}$.
	Then
	\begin{equation}\label{eq:Delta-u-L10}
		\Delta(u)
		=
		\Delta(L_{1,0})
		=
		\sum_{k=p}^{q}c_{1,k}L_{1,k}.
	\end{equation}
	By locality and the description of \(\operatorname{Der}(\B)\), there exist
	\[
	v=\sum_{(\gamma,k)\in S}b_{\gamma,k}L_{\gamma,k}\in \B,
	\qquad d\in\mathbb C,
	\]
	such that
	\[
	\Delta(u)=[v,u]+dD_0(u),
	\] 
	where \(S\subseteq\mathbb Z\times\mathbb Z_{\geq0}\) is finite and $b_{\gamma,k}\in\C$.
 Using the defining bracket of \(\B\), we obtain
	\begin{equation}\label{eq:Delta-u-local-L10}
		\begin{aligned}
			\Delta(u)
			=
			\sum_{(\gamma,k)\in S}
			\bigl(k+1-\gamma\bigr)b_{\gamma,k}L_{\gamma+1,k}
			+
			\sum_{(\gamma,k)\in S}
			\bigl((k+1)\alpha-\gamma\bigr)
			b_{\gamma,k}L_{\gamma+\alpha,k}.
		\end{aligned}
	\end{equation}
	
	\medskip
	\noindent
	\textbf{Claim.}
	If \(\gamma<0\), then
	$b_{\gamma,k}=0$ for every $k\in\mathbb Z_{\geq0}$.
	
	Suppose that \(b_{\gamma,k}\neq0\) for some \(\gamma<0\).  Comparing the coefficient
	of \(L_{\gamma+s(1-\alpha)+1,k}\) in
	\eqref{eq:Delta-u-L10} and \eqref{eq:Delta-u-local-L10} gives
\[
\begin{aligned}
	0 &= (k+1-(\gamma+s(1-\alpha)))b_{\gamma+s(1-\alpha), k}
	 + ((k+1)\alpha-(\gamma+(s+1)(1-\alpha)))b_{\gamma+(s+1)(1-\alpha), k}.
\end{aligned}
\]
Since $\gamma < 0$, we have $k+1-(\gamma+s(1-\alpha))\neq0$ and $(k+1)\alpha-(\gamma+(s+1)(1-\alpha))\neq0$. Hence,	
	\[
	b_{\gamma+s(1-\alpha),k}\neq0
	\quad\Longrightarrow\quad
	b_{\gamma_{s+1}\gamma+(s+1)(1-\alpha),k}\neq0.
	\]
	Starting from \(b_{\gamma,k}\neq0\), induction yields
	\(b_{\gamma+s(1-\alpha),k}\neq0\) for every \(s\geq0\), which
	contradicts the finiteness of \(v\). This proves the claim.
	
	Now fix \(k\in\{p,\ldots,q\}\). Comparing the coefficient of \(L_{1,k}\)
	in \eqref{eq:Delta-u-L10} and \eqref{eq:Delta-u-local-L10}, we obtain
	\begin{equation}
		c_{1,k}
		=
		(k+1)b_{0,k}
		+
		\bigl((k+2)\alpha-1\bigr)b_{1-\alpha,k}.
	\end{equation}
	Since \(1-\alpha<0\), the claim gives \(b_{1-\alpha,k}=0\). Therefore
	\begin{equation}\label{eq:c1k-a0k}
	c_{1,k}=(k+1)b_{0,k}.
	\end{equation}
	
For \(s\geq0\), compare the
	coefficient of
$L_{(s+1)\alpha-s,k}$ in \eqref{eq:Delta-u-L10} and
	\eqref{eq:Delta-u-local-L10}. Since
	\((s+1)\alpha-s\geq\alpha>1\), this basis vector does not occur in \eqref{eq:Delta-u-L10}. Hence
	\begin{equation}\label{eq:positive-propagation-L10}
		\begin{aligned}
			0=
			\bigl(k+1-(s+1)(\alpha-1)\bigr)
			b_{(s+1)(\alpha-1),k}
			+
			\bigl((k+1)\alpha-s(\alpha-1)\bigr)
			b_{s(\alpha-1),k}.
		\end{aligned}
	\end{equation}
	
	The first scalar coefficient is nonzero because
	\[
	k+1-(s+1)(\alpha-1)
	\leq q+1-(\alpha-1)<0.
	\]
	The second coefficient is also nonzero. Indeed, if
$(k+1)\alpha=s(\alpha-1)$, then \(\alpha-1|(k+1)\alpha\). Since
	\(\gcd(\alpha-1,\alpha)=1\), it follows that
	\(\alpha-1\mid k+1\), contradicting
	\[
	0<k+1\leq q+1<\alpha-1.
	\]
	
	Thus \eqref{eq:positive-propagation-L10} implies
	\[
	b_{s(\alpha-1),k}\neq0
	\quad\Longrightarrow\quad
	b_{(s+1)(\alpha-1),k}\neq0.
	\]
	If \(b_{0,k}\neq0\), induction would produce infinitely many nonzero
	coefficients \(b_{s(\alpha-1),k}\), contradicting the finiteness of \(v\).
	Hence \(b_{0,k}=0\). By (\ref{eq:c1k-a0k}), we obtain $	c_{1,k}=0$ for all $p\leq k\leq q$.
	This completes the proof.
\end{proof}

\bl{lem6}
Let $\Delta$ be a local derivation of $\B$. Replacing
$\Delta$ by $\Delta-D$ for a suitable
$D\in\operatorname{Der}(\B)$ if necessary, we may assume that
\[
\Delta(L_{0,i})=\Delta(L_{\alpha,0})=0,
\qquad i\in\mathbb Z_{\ge0},\
\alpha\in\mathbb Z.
\]
\el

\begin{proof}
	By Lemma~\ref{lem5}, after subtracting a
	suitable derivation, we may assume that
	\[
	\Delta(L_{0,i})=\Delta(L_{\alpha,0})=0,
	\quad i\in\mathbb Z_{\geq0},
	\ \alpha\in\mathbb Z_{\geq1}.
	\]
	
	Fix $\alpha\in\mathbb Z_{\leq-1}$. By Lemma~\ref{lem3}, there exist integers $0\leq p\leq q$ and $c_{\alpha,k}\in\mathbb C$ such that
	\[
	\Delta(L_{\alpha,0})
	=
	\sum_{k=p}^{q}c_{\alpha,k}L_{\alpha,k}.
	\]

Choose a prime $\ell>q+1-\alpha$ and set
$\beta=\ell+\alpha$. Then $\beta>q+1$, so
$\Delta(L_{\beta,0})=0$. Let
 $u=L_{\alpha,0}+L_{\beta,0}$.
		Then
		\begin{equation}\label{eq:negative-horizontal-linear}
			\Delta(u)
			=
			\Delta(L_{\alpha,0})
			=
			\sum_{k=p}^{q}c_{\alpha,k}L_{\alpha,k}.
		\end{equation}
By locality and Lemma~\ref{lem-der}, there exist
$v\in \B$ and $d\in\mathbb C$, such that
$
\Delta(u)=[v,u]+dD_0(u).
$
Write
\[
v=\sum_{(\gamma,k)\in S}b_{\gamma,k}L_{\gamma,k},
\]
where $S\subseteq\mathbb Z\times\mathbb Z_{\geq0}$ is finite and $b_{\gamma,k}\in\C$.  This together with 	\eqref{eq:negative-horizontal-linear} gives		
		\begin{equation}\label{eq:negative-horizontal-local}
			\begin{aligned}
				\sum_{k=p}^{q}c_{\alpha,k}L_{\alpha,k}
				=
				\sum_{(\gamma,k)\in S}
				\bigl((k+1)\alpha-\gamma\bigr)
				b_{\gamma,k}L_{\gamma+\alpha,k}
				+
				\sum_{(\gamma,k)\in S}
				\bigl((k+1)\beta-\gamma\bigr)
				b_{\gamma,k}L_{\gamma+\beta,k}.
			\end{aligned}
		\end{equation}
		
		Fix $k\in\{p,\ldots,q\}$. Comparing the coefficient of $L_{\alpha,k}$ on both sides of
		\eqref{eq:negative-horizontal-local}, we obtain
		\begin{equation}\label{eq:negative-horizontal-main}
			c_{\alpha,k}
			=
			(k+1)\alpha\,b_{0,k}
			+
			\bigl((k+2)\beta-\alpha\bigr)b_{-\ell,k}.
		\end{equation}
 We show that both coefficients
\(b_{0,k}\) and \(b_{-\ell,k}\) vanish.	
		
		\medskip
		\noindent
		\textbf{Claim 1.}
		One has $b_{0,k}=0$.
		
		Suppose, to the contrary, that $b_{0,k}\neq0$. For every $s\geq1$,
		comparing the coefficient of $L_{\alpha+s\ell,k}$ on both sides of \eqref{eq:negative-horizontal-local} gives
		\begin{equation}\label{eq:positive-P-chain}
			0
			=
			\bigl((k+1)\alpha-s\ell\bigr)b_{s\ell,k}
			+
			\bigl((k+1)\beta-(s-1)\ell\bigr)b_{(s-1)\ell,k}.
		\end{equation}
		
		The first scalar coefficient is negative and hence nonzero. The second is
		also nonzero. Indeed, if
		\[
		(k+1)\beta=(s-1)\ell,
		\]
		then $\ell\mid(k+1)\beta$. Since $\ell$ is prime and $0<\beta<\ell$, we have
		$\gcd(\ell,\beta)=1$, and therefore $\ell\mid k+1$. This is impossible because
		\[
		1\leq k+1\leq q+1<\ell.
		\]
		It follows from \eqref{eq:positive-P-chain} that
		\[
		b_{(s-1)\ell,k}\neq0
		\quad\Longrightarrow\quad
		b_{s\ell,k}\neq0.
		\]
		Starting with $b_{0,k}\neq0$, induction gives \(b_{s\ell,k}\neq0\) for every \(s\geq1\),
		contradicting the finiteness of $v$. Claim~1 is proved.

\medskip
\noindent
\textbf{Claim 2.}
One has \(b_{-\ell,k}=0\).

Suppose, to the contrary, that \(b_{-\ell,k}\neq0\). For \(s\geq1\),
comparing the coefficient of \(L_{\alpha-s\ell,k}\) on both sides of \eqref{eq:negative-horizontal-local}, we obtain
\begin{equation}\label{eq:negative-ell-chain}
	0
	=
	\bigl((k+1)\alpha+s\ell\bigr)b_{-s\ell,k}
	+
	\bigl((k+1)\beta+(s+1)\ell\bigr)
	b_{-(s+1)\ell,k}.
\end{equation}

The second factor in \eqref{eq:negative-ell-chain} is positive and hence
nonzero. The first factor is also nonzero. Otherwise,
\[
s\ell=-(k+1)\alpha.
\]
Since \(0<-\alpha<\ell\) and \(\ell\) is prime, we have
\(\gcd(\ell,-\alpha)=1\). Thus \(\ell\mid k+1\), contradicting
\(1\leq k+1\leq q+1<\ell\).

Consequently,
\[
b_{-s\ell,k}\neq0
\quad\Longrightarrow\quad
b_{-(s+1)\ell,k}\neq0.
\]
Starting with \(b_{-\ell,k}\neq0\), induction yields
\(b_{-s\ell,k}\neq0\) for every \(s\geq1\). Again, this contradicts the
finiteness of \(v\). Therefore \(b_{-\ell,k}=0\), proving Claim~2.

Substituting Claims~1 and~2 into \eqref{eq:negative-horizontal-main}, we obtain
\(c_{\alpha,k}=0\) for every \(p\leq k\leq q\). Hence
\(\Delta(L_{\alpha,0})=0\), and the proof is complete.
	\end{proof}

After the preceding normalizations, we may assume that
\(\Delta(L_{0,i})=0\) for all \(i\in\mathbb Z_{\geq0}\) and
\(\Delta(L_{\alpha,0})=0\) for all \(\alpha\in\mathbb Z\).
We now determine the action of \(\Delta\) on the remaining elements
\(L_{\alpha,i}\). Lemma \ref{lem7} restricts the possible terms occurring in
\(\Delta(L_{\alpha,i})\), Lemma \ref{lem8} eliminates all off-diagonal terms,
and Lemma \ref{lem9} shows that the remaining diagonal coefficients vanish as
well.

\bl{lem7}
Let $\Delta$ be a local derivation of $\B$ satisfying $\Delta(L_{0, i})=\Delta(L_{\alpha,0})=0$ for all $\alpha\in \Z$, $i\in \Z_{\geq0}$. Then, for every $\alpha,i\in\mathbb Z_{\geq1}$, there exist a finite subset
$K\subseteq\mathbb Z_{\geq2}$ and coefficients
$c_{\alpha,i},c_{s\alpha,\,s(i+1)-2}\in\mathbb C$, such that
\[
\Delta(L_{\alpha,i})
=
c_{\alpha,i}L_{\alpha,i}
+
\sum_{s\in K}
c_{s\alpha,\,s(i+1)-2}L_{s\alpha,\,s(i+1)-2}.
\]
\el
\begin{proof}
Fix \(\alpha,i\in\mathbb Z_{\geq1}\). Write
	\begin{equation}\label{lem-3.9-1}
	\Delta(L_{\alpha,i})
	=
	\sum_{\mu=a}^{b}\sum_{r=m}^{n}c_{\mu,r}L_{\mu,r},
	\end{equation}
	where integers \(a\leq b\), \(0\leq m\leq n\), and $c_{\mu,r}\in\C$. Put
	\(
	M=\max\{|a|,|b|\},
	\)
	and choose an integer $\beta$ such that
	\(
	\beta>2M+(n+2)\alpha.
	\)
	Set
	\(
	u=L_{\alpha,i}+L_{\beta,0}.
	\)
	Since \(\Delta(L_{\beta,0})=0\), we have
	$
	\Delta(u)=\Delta(L_{\alpha,i}).
	$
By locality and Lemma {\ref{lem-der}}, there exist
\[
v=\sum_{(\gamma,k)\in S}b_{\gamma,k}L_{\gamma,k}\in \B,
\qquad d\in\mathbb C,
\]
such that
\[
\Delta(L_{\alpha,i})
=[v,L_{\alpha,i}]+[v,L_{\beta,0}]
+dD_0(L_{\alpha,i}),
\]
	where \(S\subset\mathbb Z\times\mathbb Z_{\geq0}\) is finite and $b_{\gamma,k}\in\C$.
	Using the defining bracket of $\B$, we obtain
	\begin{equation}\label{lem-3.9-2}
	\begin{aligned}
		\Delta(L_{\alpha,i})
		&=
		\sum_{(\gamma,k)\in S}
		\bigl((k+1)\alpha-(i+1)\gamma\bigr)
		b_{\gamma,k}L_{\gamma+\alpha,k+i}        \\
		&\quad+
		\sum_{(\gamma,k)\in S}
		\bigl((k+1)\beta-\gamma\bigr)
		b_{\gamma,k}L_{\gamma+\beta,k}
		+
		diL_{\alpha,i}.
	\end{aligned}
	\end{equation}

Let \(L_{\mu,r}\) be a non-diagonal term of
\(\Delta(L_{\alpha,i})\); that is,
\[
c_{\mu,r}\neq0,
\qquad
(\mu,r)\neq(\alpha,i).
\]

\medskip
\noindent
\textbf{Claim 1.}
Every non-diagonal term of \(\Delta(L_{\alpha,i})\) comes from the first sum
in \((\ref{lem-3.9-2})\).

Suppose that \(L_{\mu,r}\) , \((\mu,r)\neq(\alpha,i)\), arises from the second sum. Then
\[
b_{\mu-\beta,r}\neq0.
\]
For \(t\geq0\), set
\[
\gamma_t=\mu+t\alpha-(t+1)\beta,
\qquad
k_t=r+ti.
\]
We prove inductively that
\[
b_{\gamma_t,k_t}\neq0,
\qquad t\geq0.
\]
The assertion holds for \(t=0\). Assume that
\(b_{\gamma_t,k_t}\neq0\). Since
\[
\gamma_t+\alpha
=
\mu+(t+1)(\alpha-\beta)<a,
\]
the basis element
\(L_{\gamma_t+\alpha,k_t+i}\) does not occur in
(\ref{lem-3.9-1}). Comparing its coefficient in
(\ref{lem-3.9-1}) and (\ref{lem-3.9-2}), we obtain
\[
0=
\bigl((k_t+1)\alpha-(i+1)\gamma_t\bigr)b_{\gamma_t,k_t}
+
\bigl((k_{t+1}+1)\beta-\gamma_{t+1}\bigr)
b_{\gamma_{t+1},k_{t+1}}.
\]
Moreover,
\[
(k_t+1)\alpha-(i+1)\gamma_t
=
(i+1)(t+1)\beta+(r+1-t)\alpha-(i+1)\mu
\neq0.
\]
Hence \(b_{\gamma_{t+1},k_{t+1}}\neq0\). This produces infinitely
many nonzero coefficients of \(v\), a contradiction. Therefore Claim~1
holds.

\medskip
\noindent
\textbf{Claim 2.}
Let \(L_{\mu,r}\) be a non-diagonal term of
\(\Delta(L_{\alpha,i})\), and put 
$
\gamma=\mu-\alpha$, $k=r-i$.
Then \(k\geq0\), \(b_{\gamma,k}\neq0\), and there exists
\(t\in\mathbb Z_{\geq0}\) such that
$
k-ti\geq0
$
and
\begin{equation}\label{lem-3.9-3}
	(k-ti+1)\beta=\gamma+t(\beta-\alpha).
\end{equation}

By Claim~1, \(L_{\mu,r}\) comes from the first sum in (\ref{lem-3.9-2}).  Hence
\(\mu=\gamma+\alpha\), \(r=k+i\), and
\(b_{\gamma,k}\neq0\) for some \(\gamma,k\).  
For \(t\geq0\), set
\[
\gamma_t=\gamma+t(\beta-\alpha),
\qquad
k_t=k-ti.
\]
Suppose that \(k_t\geq0\) and \(b_{\gamma_t,k_t}\neq0\). Since
\[
\gamma_t+\beta
=
\mu+(t+1)(\beta-\alpha)>b,
\]
the element \(L_{\gamma_t+\beta,k_t}\) does not occur in
(\ref{lem-3.9-1}). Comparing its coefficient gives
\[
0=
\bigl((k_{t+1}+1)\alpha-(i+1)\gamma_{t+1}\bigr)
b_{\gamma_{t+1},k_{t+1}}
+
\bigl((k_t+1)\beta-\gamma_t\bigr)b_{\gamma_t,k_t},
\]
where the first term is omitted when \(k_{t+1}<0\).

By the choice of \(\beta\),
\[
(k_{t+1}+1)\alpha-(i+1)\gamma_{t+1}\neq0.
\]
Thus, whenever
\[
(k_t+1)\beta-\gamma_t\neq0
\quad\text{and}\quad
k_{t+1}\geq0,
\]
we obtain \(b_{\gamma_{t+1},k_{t+1}}\neq0\).
Since the second indices \(k_t\) decrease by \(i\), this process cannot
continue indefinitely. Therefore, for some \(t\geq0\) with \(k_t\geq0\),
we must have
\[
(k_t+1)\beta-\gamma_t=0,
\]
which is precisely \eqref{lem-3.9-3}. This proves Claim~2.

\medskip
\noindent
\textbf{Claim 3.}
Every non-diagonal term of \(\Delta(L_{\alpha,i})\) has the form
\(
L_{s\alpha,\,s(i+1)-2}
\)
for some \(s\in\mathbb Z_{\geq2}\).

Let \(L_{\mu,r}\) be a non-diagonal term. By Claim~2, there exists
\(t\geq0\) such that
\[
(k-ti+1)\beta
=
\gamma+t(\beta-\alpha),
\]
where \(\gamma=\mu-\alpha\) and \(k=r-i\). Equivalently,
\begin{equation}
	\bigl(k+1-t(i+1)\bigr)\beta
	=
	\mu-(t+1)\alpha.
	\label{lem3.7-index}
\end{equation}
Since \(k-ti\geq0\), we have \(t\leq k\leq n\). Hence
\[
\bigl|\mu-(t+1)\alpha\bigr|
\leq M+(n+1)\alpha
<\beta.
\]
The left-hand side of \eqref{lem3.7-index} is an integral multiple of
\(\beta\). Therefore
\[
k+1=t(i+1),
\qquad
\mu=(t+1)\alpha.
\]
Moreover,
\[
r=k+i
=t(i+1)+i-1
=(t+1)(i+1)-2.
\]
Since \(k\geq0\), the equality \(k+1=t(i+1)\) implies \(t\geq1\).
Setting \(s=t+1\), we obtain
\[
\mu=s\alpha,
\qquad
r=s(i+1)-2,
\qquad
s\geq2.
\]
This proves Claim~3.

By Claims1-3, the lemma is proved.
\end{proof}


\bl{lem8}
Let \(\Delta\) be a local derivation of \(B\) satisfying $\Delta(L_{0, i})=\Delta(L_{\alpha,0})=0$ for all $\alpha\in\Z, i\in \Z_{\geq0}$. Then, for every $\alpha\in \Z$ and $i\in \Z_{\geq0}$, there exist coefficients $c_{\alpha, i}\in\C$ such that
\[
\Delta(L_{\alpha, i})=c_{\alpha, i}L_{\alpha, i}.
\]
\el
\begin{proof}
	The assertion is immediate if \(\alpha=0\) or \(i=0\), by the
	assumptions. We first assume that
	\(\alpha,i\in\mathbb Z_{\geq1}\).
	
	By Lemma~\ref{lem7}, we may assume
	\begin{equation}\label{eq:index-restricted-form}
		\Delta(L_{\alpha,i})
		=
		c_{\alpha,i}L_{\alpha,i}
		+
		\sum_{s\in K}
		c_{s\alpha,\,s(i+1)-2}L_{s\alpha,\,s(i+1)-2},
	\end{equation}
	where \(K\subseteq\mathbb Z_{\geq2}\) is finite and 
	\(c_{\alpha,i},c_{s\alpha,\,s(i+1)-2}\in\mathbb C\).
	By locality and Lemma~\ref{lem-der}, there exist
	$
	v=\sum_{(\gamma,k)\in S}
	b_{\gamma,k}L_{\gamma,k}\in \B$ and $d\in\mathbb C$,
	 such that
	\[
	\Delta(L_{\alpha,i})
	=
	[v,L_{\alpha,i}]
	+
	dD_0(L_{\alpha,i}),
	\] 
	where \(S\subseteq\mathbb Z\times\mathbb Z_{\geq0}\) is finite and $b_{\gamma,k}\in\C$.
	Using the defining bracket of \(\B\), we obtain
	\begin{equation}\label{eq:local-value-single-basis}
		\Delta(L_{\alpha,i})
		=
		\sum_{(\gamma,k)\in S}
		\bigl((k+1)\alpha-(i+1)\gamma\bigr)
		b_{\gamma,k}L_{\gamma+\alpha,k+i}
		+
		diL_{\alpha,i}.
	\end{equation}
	
	Fix \(s\in K\). Then the coefficient of
	\(L_{s\alpha,\,s(i+1)-2}\) in
	\eqref{eq:local-value-single-basis} is zero. Comparing it with
	the corresponding coefficient in
	\eqref{eq:index-restricted-form}, we obtain \(c_{s\alpha,\,s(i+1)-2}=0\).
	Therefore
		\begin{equation}\label{lem-3.9-35}
	\Delta(L_{\alpha,i})
	=
	c_{\alpha,i}L_{\alpha,i},
	\qquad
	\alpha,i\in\mathbb Z_{\geq1}.
		\end{equation}
	
	It remains to consider $\alpha\in\mathbb Z_{\leq-1}$. Put
	$\beta=-\alpha$. By Corollary \ref{coro-B}, $\widetilde{\Delta}=\sigma_0\Delta\sigma_0^{-1}$ is a local
	derivation and satisfies the same assumptions. Applying the positive case to
	$\widetilde{\Delta}(L_{\beta,i})$, and then using the relation
	$
	\widetilde{\Delta}(L_{\beta,i})
	=
	-\sigma_0\bigl(\Delta(L_{\alpha,i})\bigr),
	$
	together with
	\eqref{lem-3.9-35}, we obtain
	\[
	\Delta(L_{\alpha,i})
	=
	c_{\alpha,i}L_{\alpha,i},
	\qquad
	\alpha\in\mathbb Z_{\leq-1},
	\
	i\in\mathbb Z_{\geq1}.
	\]
	
	This completes the proof.
\end{proof}

\bl{lem9}
Let $\Delta$ be a local derivation of $\B$ satisfying $\Delta(L_{0, i})=\Delta(L_{\alpha,0})=0$ for all $\alpha\in \Z$ and $i\in \Z_{\geq 0}$. Then $\Delta=0$.
\el

\begin{proof}
	By Lemma~\ref{lem8}, for
	\(\alpha\in\mathbb Z\) and \(i\in\mathbb Z_{\geq1}\),
	there exists \(c_{\alpha,i}\in\mathbb C\) such that
	\[
	\Delta(L_{\alpha,i})
	=
	c_{\alpha,i}L_{\alpha,i}.
	\]
	It suffices to prove that
	$
	c_{\alpha,i}=0.
	$
	
	We first assume that
	\(\alpha,i\in\mathbb Z_{\geq1}\).
	Choose integers \(j,\beta\) satisfying
	$j>i$ and $\beta>\alpha$.
	Set
	\(
	u=L_{\alpha,i}+L_{0,j}+L_{\beta,0}.
	\)
	Since
	\(\Delta(L_{0,j})=\Delta(L_{\beta,0})=0\),
	we have
	\(
	\Delta(u)=c_{\alpha,i}L_{\alpha,i}.
	\)
	By locality and Lemma \ref{lem-der}, there exist
	\[
	v=\sum_{(\gamma,k)\in S}b_{\gamma,k}L_{\gamma,k}\in \B ,\quad d\in\mathbb C,\]
	where \(S\subseteq\mathbb Z\times\mathbb Z_{\geq0}\) is finite and $b_{\gamma,k}\in\C$, such that
	\(
	\Delta(u)=[v,u]+dD_0(u).
	\)
	Then we obtain
	\begin{equation}\label{3.36}
	\begin{aligned}
		c_{\alpha,i}L_{\alpha,i}
		&=
		\sum_{(\gamma,k)\in S}
		((k+1)\alpha-(i+1)\gamma)
		b_{\gamma,k}L_{\gamma+\alpha,k+i}
		-\sum_{(\gamma,k)\in S}
		(j+1)\gamma b_{\gamma,k}L_{\gamma,k+j}\\
		&\quad
		+\sum_{(\gamma,k)\in S}
		((k+1)\beta-\gamma)b_{\gamma,k}L_{\gamma+\beta,k}
		+d(iL_{\alpha,i}+jL_{0,j}).
	\end{aligned}
	\end{equation}
	
	\medskip
	\noindent
	\textbf{Claim.}
	If \(\gamma<0\), then
	\(b_{\gamma,k}=0\) for every \(k\in\mathbb Z_{\geq0}\).
	
	Suppose, to the contrary, that \(b_{\gamma,k}\neq0\) for some
	\(\gamma<0\). Comparing the coefficient of \(L_{\gamma,k+j}\) on both sides of (\ref{3.36}), we obtain
	\[
		0
		=
		\bigl((k+j+2)\alpha-(i+1)\gamma\bigr)
		b_{\gamma-\alpha,k+j-i}
		-(j+1)\gamma b_{\gamma,k}
		+
		\bigl((k+j+2)\beta-\gamma\bigr)
		b_{\gamma-\beta,k+j}.
	\]
	Since 
	\[
	(k+j+2)\alpha-(i+1)\gamma>0,\quad -(j+1)\gamma>0,
	\quad
	(k+j+2)\beta-\gamma>0,
	\]
	all three coefficients are nonzero. Consequently, $b_{\gamma-\alpha,k+j-i}\neq0$ or $b_{\gamma-\beta,k+j}\neq0$. Since $\gamma-\alpha, \gamma-\beta<0$,  repeating the same argument inductively, we obtain infinitely many
	nonzero coefficients of \(v\), a contradiction. This proves the Claim. 
	
	Comparing the coefficient of \(L_{0,j}\) on both sides of (\ref{3.36}) and using the Claim, we obtain
	$dj=0$.
	 Since \(j>0\), it follows that $d=0$. Finally, comparing the coefficients of $L_{\alpha, i}$ and $L_{\beta,0}$ on both sides of (\ref{3.36}) yields
	\[c_{\alpha,i}
	=
	\alpha b_{0,0}
	+
	\bigl((i+2)\beta-\alpha\bigr)b_{\alpha-\beta,i}
	+
	di,\quad
	\beta b_{0,0}  = 0.\]
	Since   $\alpha-\beta<0$, the claim gives
	$
	b_{\alpha-\beta,i}=0.
	$
	Hence
	\[
	c_{\alpha,i}
	=
	\alpha b_{0,0}=0.
	\]
	
	It remains to consider
	$
	\alpha\in\mathbb Z_{\le-1}.
	$
	For
	\(\alpha\le-1\),
	let
	$
	\widetilde{\Delta}
	=
	\sigma_0\Delta\sigma_0^{-1}.
	$
	By Corollary \ref{coro-B},
	\(\widetilde{\Delta}\)
	is again a local derivation satisfying the same assumptions.
	Applying the positive case to
	\(\widetilde{\Delta}(L_{-\alpha,i})\)
	and using
	\[
	\widetilde{\Delta}(L_{-\alpha,i})
	=
	-\sigma_0(\Delta(L_{\alpha,i})),
	\]
	we conclude that
	\[
	\Delta(L_{\alpha,i})=0,
	\qquad
	\alpha\in\mathbb Z_{\le-1},
	\
	i\in\mathbb Z_{\ge1}.
	\]
	Hence
	$
	\Delta=0.
	$
\end{proof}

\begin{proof}[Proof of Theorem~1.1]
	Let $\Delta\in \operatorname{LDer}(\B)$. By Lemma \ref{lem6}, after replacing
	\(\Delta\)
	by
	\(\Delta-D\)
	for a suitable
	\(D\in\operatorname{Der}(\B)\),
	we may assume that
	\[
	\Delta(L_{0,i})=\Delta(L_{\alpha,0})=0,
	\qquad
	i\in\mathbb Z_{\geq0},
	\ \alpha\in\mathbb Z.
	\]
	By Lemma \ref{lem9}, it follows that
	$
	\Delta=0.
	$
	Therefore every local derivation of \(B\) is a derivation, that is,
	\[
	\operatorname{LDer}(\B)
	=
	\operatorname{Der}(\B),
	\]
	which completes the proof of Theorem~1.1.
\end{proof}

\end{document}